\pdfoutput=1
\documentclass{amsart}
\usepackage{todonotes}
\usepackage[T1]{fontenc}
\usepackage[utf8]{inputenc}
\usepackage{lmodern}
\usepackage{tikz-cd}
\usepackage{quiver}
\usepackage{amssymb,amsmath,amsthm,mathrsfs,mathtools}
\usepackage[dvipsnames]{xcolor}
\usepackage[unicode,colorlinks]{hyperref}
\hypersetup{
  linkcolor=BrickRed,
  citecolor=Green,
  urlcolor=NavyBlue
}
\usepackage[nameinlink,capitalise,noabbrev]{cleveref}

\newcommand{\RadId}{\mathrm{RadId}}
\newcommand{\fp}{\mathfrak{p}}
\newcommand{\Frm}{\mathbf{Frm}}

\newcommand{\pt}{\mathrm{pt}}

\newcommand{\Nak}{\mathbf{Nak}}
\newcommand{\bbZ}{\mathbb{Z}}

\DeclareMathOperator{\Spec}{Spec}
\DeclareMathOperator{\Hom}{Hom}

\newcommand{\RadFg}{\mathrm{RadFg}}

\newcommand{\op}{\mathrm{op}}

\AddToHook{env/Thm/begin}{     \crefalias{equation}{Thm}}
\AddToHook{env/Exa/begin}{     \crefalias{equation}{Exa}}
\AddToHook{env/Prop/begin}{    \crefalias{equation}{Prop}}
\AddToHook{env/Lem/begin}{     \crefalias{equation}{Lem}}
\AddToHook{env/Cor/begin}{     \crefalias{equation}{Cor}}
\AddToHook{env/Conj/begin}{    \crefalias{equation}{Conj}}
\AddToHook{env/Rem/begin}{     \crefalias{equation}{Rem}}
\AddToHook{env/Def/begin}{     \crefalias{equation}{Def}}

\numberwithin{equation}{section}
\swapnumbers

\newtheorem{Thm}[equation]{Theorem}
\newtheorem*{Thm*}{Theorem}
\newtheorem{Prop}[equation]{Proposition}
\newtheorem{Lem}[equation]{Lemma}
\newtheorem{Cor}[equation]{Corollary}

\theoremstyle{remark}
\newtheorem{Def}[equation]{Definition}
\newtheorem{Not}[equation]{Notation}
\newtheorem{Exa}[equation]{Example}
\newtheorem{Rem}[equation]{Remark}

\Crefname{Thm}{Theorem}{Theorems}
\Crefname{Prop}{Proposition}{Propositions}
\Crefname{Def}{Definition}{Definitions}
\Crefname{Lem}{Lemma}{Lemmas}
\Crefname{Cor}{Corollary}{Corollaries}
\Crefname{Rem}{Remark}{Remarks}
\Crefname{Exa}{Example}{Examples}

\title[The Nakaoka spectrum, point-free]{A point-free approach to the Nakaoka spectrum of a Tambara functor}
\date{\today}
\author{Drew Heard }

\address{Drew Heard, Department of Mathematical Sciences, Norwegian University of Science and Technology, Trondheim}
\email{drew.k.heard@ntnu.no}
\urladdr{\href{https://folk.ntnu.no/drewkh/}{https://folk.ntnu.no/drewkh/}}
\begin{document}

\begin{abstract}
For $G$ a finite group and $T$ a $G$-Tambara functor, we construct the frame
$\RadId_G(T)$ of radical Tambara ideals and show that its points are the Nakaoka primes. We show that this frame is spatial and coherent,  and deduce that the Nakaoka spectrum is a spectral space, recovering a recent result of Chan and Spitz.
\end{abstract}

\maketitle

\tableofcontents

\section{Introduction}\label{sec:intro}

The Zariski spectrum of a commutative ring is one of the foundational
constructions of algebraic geometry. An elegant approach, due to Joyal
\cite{Joyal1975ChevalleyTarski}, constructs the spectrum entirely within the framework of point-free topology. To do so, one builds the frame $\RadId(R)$ of radical
ideals, observes that its points are exactly the prime ideals, and deduces
spatiality and spectrality from the lattice-theoretic structure. The purpose of
this paper is to carry out the analogous construction for the Nakaoka spectrum of
a $G$-Tambara functor.

Tambara functors were introduced by Tambara \cite{Tambara1993Multiplicative} as
the natural algebraic structure encoding the interaction of restriction,
transfer, and norm maps in equivariant algebra. They play the role of commutative
rings in $G$-equivariant homotopy theory; just as $\pi_0$ of a commutative ring
spectrum is a commutative ring, $\pi_0$ of a genuine $G$-commutative ring
spectrum is a Tambara functor
\cite{HillMazur2019Equivariant,strickland2012tambarafunctors}. 

The development of an ``equivariant commutative algebra'' of Tambara functors,
including their ideal theory, prime spectra, and associated geometry, is a
subject of active current interest
\cite{Nakaoka2012Ideals,chan2025tambaraaffineline,ChanSpitz2026RadicalsNilpotents}.
One motivation for a point-free approach is the prospect of an enhanced Balmer
spectrum for categories of $G$-spectra that incorporates norm functors and hence
Tambara structure; the frame-theoretic methods developed here may serve as an
algebraic model for such a construction.

The prime spectrum of a Tambara functor was introduced by Nakaoka
\cite{Nakaoka2012Ideals}. A proper Tambara ideal $P\trianglelefteq T$ is \emph{prime}
if $IJ\subseteq P$ implies $I\subseteq P$ or $J\subseteq P$ for all ideals
$I,J$, and the Nakaoka spectrum $\Spec_{\Nak}(T)$ is the set of primes equipped
with a Zariski-type topology. 

Our main results are as follows.

\begin{Thm*}[\Cref{thm:points-primes,thm:spatiality,cor:frame-opens}]
Let $T$ be a $G$-Tambara functor. The frame $\RadId_G(T)$ of radical Tambara
ideals is spatial, and there is a canonical homeomorphism
\[
\pt\bigl(\RadId_G(T)\bigr)\cong\Spec_{\Nak}(T).
\]
In particular, $\RadId_G(T)\cong\Omega(\Spec_{\Nak}(T))$.
\end{Thm*}
Moreover, the frame $\RadId_G(T)$ carries enough finiteness to guarantee
spectrality.
\begin{Thm*}[\Cref{thm:coherent,thm:spectrality}]
The frame $\RadId_G(T)$ is coherent, and hence the Nakaoka spectrum $\Spec_{\Nak}(T)$ is a spectral space.
\end{Thm*}
Spectrality of $\Spec_{\Nak}(T)$ was independently established by Chan and
Spitz \cite{ChanSpitz2026RadicalsNilpotents} by studying radicals and
nilpotents. Our approach is complementary and proceeds through the
lattice-theoretic structure of $\RadId_G(T)$ and gives, as a by-product, the
explicit identification of the compact opens with radical finitely generated
Tambara ideals (\Cref{rem:compact-opens}). We note, however, that we crucially
use a result of Chan and Spitz \cite[Corollary~4.7]{ChanSpitz2026RadicalsNilpotents},
namely that the radical of a Tambara ideal can be computed levelwise.

As further applications of the frame-theoretic perspective we
develop some commutative algebra of Tambara functors, for example, nilradicals, reductions, coprime decompositions, and irreducibility. 

\subsection*{Conventions}
Throughout, $G$ denotes a finite group. By ``ring'' we mean commutative ring with unit. We write $\sqrt{I}$ for the radical of an ideal (in a ring or in a Tambara functor, as context dictates).

\section{Frames and point-free topology}\label{sec:frames}

We begin with a review of point-free topology. Nothing in this section is new;
detailed treatments can be found in
\cite{Johnstone1986Stone,PicadoPultr2012Frames}.

\begin{Def}\label{def:frame}
A \emph{frame} is a complete lattice $A$ in which finite meets distribute over
arbitrary joins: for all $a\in A$ and $\{b_i\}_{i\in I}\subseteq A$,
\[
a\wedge \Big(\bigvee_{i\in I} b_i\Big)
=\bigvee_{i\in I}(a\wedge b_i).
\]
We write $\top$ and $\bot$ for the top and bottom elements.
\end{Def}

\begin{Def}\label{def:frame-hom}
A \emph{frame homomorphism} $\varphi\colon A\to B$ is a map preserving all joins
and finite meets, i.e., 
\[
\varphi\Big(\bigvee_{i\in I} a_i\Big)=\bigvee_{i\in I}\varphi(a_i),
\qquad
\varphi(a\wedge b)=\varphi(a)\wedge\varphi(b),
\qquad
\varphi(\top)=\top.
\]
Frames and frame homomorphisms form a category $\Frm$.
\end{Def}

\begin{Exa}\label{exa:opens-frame}
For a topological space $X$, the poset $\Omega(X)$ of open subsets is a frame
with joins given by unions and meets by intersections. A continuous map
$f\colon X\to Y$ induces a frame homomorphism
$f^{-1}\colon \Omega(Y)\to \Omega(X)$, so $\Omega$ defines a functor
$\mathbf{Top}\to\Frm^{\op}$.
\end{Exa}

\begin{Def}\label{def:pt-frame}
Let $2=\{0<1\}$ be the two-element frame. A \emph{point} of a frame $A$ is a
frame homomorphism $p\colon A\to 2$. The \emph{spectrum} (or \emph{point space})
of $A$ is
\[
\pt(A) \coloneqq \Hom_{\Frm}(A,2),
\]
topologised by the basis of opens
$U(a)\coloneqq\{p\in \pt(A)\mid p(a)=1\}$, $a\in A$.
\end{Def}

\begin{Def}\label{def:meet-prime}
Let $L$ be a lattice with top element $\top$. An element $q\neq \top$ is
\emph{meet-prime} if whenever $l\wedge m\le q$ we have $l\le q$ or $m\le q$.
\end{Def}

\begin{Rem}\label{rem:points-meet-prime}
There is a bijection between points of a frame $A$ and meet-prime elements
of $A$. Given a meet-prime $q\neq\top$, the map $p_q\colon A\to 2$ defined by
\[
p_q(a)=
\begin{cases}0 &\text{if } a\le q,\\
1 & \text{otherwise,}
\end{cases}
\]
is a point. Conversely, given a point $p$, the element
$q_p\coloneqq\bigvee\{a\in A\mid p(a)=0\}$ is meet-prime, and these
constructions are mutually inverse.
\end{Rem}

\begin{Rem}\label{rem:adjunction-omega-pt}
There is an adjunction $\Omega:\mathbf{Top}\rightleftarrows \Frm^{\op}:\pt$ with
$\Omega$ left adjoint to $\pt$
\cite[Theorem~4.6.3]{PicadoPultr2012Frames}. The unit at a space $X$ is the
continuous map $\eta_X\colon X\to \pt(\Omega(X))$, $x\mapsto p_x$, where
$p_x(U)=1$ if and only if $x\in U$.
\end{Rem}

\begin{Def}\label{def:spatial-frame}
The assignment $a\mapsto U(a)$ defines a canonical frame homomorphism
$\iota_A\colon A\to \Omega(\pt(A))$. A frame $A$ is \emph{spatial} if
$\iota_A$ is an isomorphism.
\end{Def}

\begin{Prop}\label{prop:spatial-separation}
A frame $A$ is spatial if and only if whenever $m\not\le n$ in $A$, there exists
a point $p\in \pt(A)$ with $p(m)=1$ and $p(n)=0$.
\end{Prop}

\begin{proof}
The map $\iota_A$ is always surjective, since every open in $\pt(A)$ is a union
of basic opens, and for any family $\{a_i\}_{i\in I}$ we have
\[
\bigcup_{i\in I} U(a_i)=U\Bigl(\bigvee_{i\in I} a_i\Bigr)
=\iota_A\Bigl(\bigvee_{i\in I} a_i\Bigr).
\]

Suppose first that $A$ is spatial, so $\iota_A$ is injective. If $m\not\le n$,
then $U(m)\not\subseteq U(n)$, as otherwise
\[
U(m)\subseteq U(n)\implies U(m)=U(m)\cap U(n)=U(m\wedge n),
\]
and injectivity of $\iota_A$ would imply $m=m\wedge n\le n$, a contradiction.
Thus there exists a point $p\in U(m)\setminus U(n)$, i.e., $p(m)=1$ and
$p(n)=0$.

Conversely, assume that whenever $m\not\le n$ there exists a point $p$ with
$p(m)=1$ and $p(n)=0$. To show that $\iota_A$ is injective, it suffices to show
that it reflects order. If $U(m)\subseteq U(n)$ and $m\not\le n$, then by
assumption there exists $p\in U(m)\setminus U(n)$, a contradiction. Hence
$U(m)\subseteq U(n)$ implies $m\le n$. Since $\iota_A$ is always
order-preserving, it is order-reflecting and therefore injective.
\end{proof}

\begin{Rem}\label{rem:sober-spatial}
A space $X$ is \emph{sober} if the unit $X\to \pt(\Omega(X))$ is a
homeomorphism. The adjunction of \Cref{rem:adjunction-omega-pt} restricts to an
(anti-)equivalence between sober spaces and spatial frames. Sober spaces have a purely
topological description, namely, they are those spaces for which every nonempty
irreducible closed subset is the closure of a unique point.
\end{Rem}

\begin{Def}\label{def:compact-element}
An element $a\in A$ is \emph{compact} if whenever
$a\le \bigvee_{i\in I} b_i$ there exists a finite $J\subseteq I$ with
$a\le \bigvee_{j\in J} b_j$. A frame $A$ is \emph{coherent} if it is generated
under arbitrary joins by its compact elements, the compact elements are closed
under finite meets, and $\top$ is compact.
\end{Def}

\begin{Prop}\label{prop:coherent-spectral}
A coherent frame is spatial, and its point space is spectral in the sense of
Hochster. Conversely, the frame of opens of any spectral space is coherent. This
gives a dual equivalence between coherent frames and spectral spaces.
\end{Prop}

\begin{proof}
See \cite[II.3.4]{Johnstone1986Stone}.
\end{proof}

\section{The Zariski spectrum of a commutative ring}\label{sec:ring-zariski}

Throughout this section $R$ denotes a commutative ring. We construct the Zariski
spectrum by building the frame of radical ideals, identifying its points with
prime ideals, and deducing spatiality and spectrality. There is of course nothing
new in this section; we include it only as a warm up for the more interesting
case of a Tambara functor. The ideas go back to Joyal
\cite{Joyal1975ChevalleyTarski}; see also \cite{Johnstone1986Stone} or
Section~1.3 of \cite{KockPitsch2017Hochster}.

\begin{Not}\label{not:radid-ring}
Let $\RadId(R)$ denote the set of radical ideals of $R$, ordered by inclusion.
\end{Not}

\begin{Lem}\label{lem:ring-radical-product}
For any ideals $I,J$ of $R$,
\[
\sqrt{IJ}=\sqrt{I}\cap\sqrt{J}.
\]
In particular, if $I$ and $J$ are radical then $\sqrt{IJ}=I\cap J$.
\end{Lem}

\begin{proof}
Since $IJ\subseteq I\cap J$, we have
$\sqrt{IJ}\subseteq\sqrt{I\cap J}\subseteq\sqrt{I}\cap\sqrt{J}$. For the
reverse, let $x\in\sqrt{I}\cap\sqrt{J}$. Then $x^m\in I$ and $x^n\in J$ for
some $m,n\ge 1$, so $x^{m+n}\in IJ$, giving $x\in\sqrt{IJ}$.
\end{proof}

The following is well-known, although we are unsure where the result was first
proved.

\begin{Prop}\label{prop:radid-ring-frame}
$\RadId(R)$ is a frame, with meets given by intersections and joins by
$\bigvee_\lambda I_\lambda \coloneqq \sqrt{\sum_\lambda I_\lambda}$.
\end{Prop}

\begin{proof}
Intersections of radical ideals are radical, so $\RadId(R)$ has arbitrary meets.

Let $\{I_\lambda\}_\lambda$ be a family of radical ideals and set
$J=\sum_\lambda I_\lambda$. Then $\sqrt{J}$ is radical and contains each
$I_\lambda$, hence is an upper bound. If $K\in\RadId(R)$ is any other upper
bound, then $J\subseteq K$, so $\sqrt{J}\subseteq \sqrt{K}=K$. Thus $\sqrt{J}$
is the join.

It remains to check distributivity, i.e.,  for $K\in\RadId(R)$ and
$\{J_\lambda\}_\lambda\subseteq\RadId(R)$,
\[
K\cap \sqrt{\sum_\lambda J_\lambda} = \sqrt{\sum_\lambda (K\cap J_\lambda)}.
\]
Since $K$ is radical, \Cref{lem:ring-radical-product} gives
$K\cap \sqrt{\sum_\lambda J_\lambda}=\sqrt{K\cdot \sum_\lambda J_\lambda}$.
Now $K\cdot \sum_\lambda J_\lambda = \sum_\lambda (KJ_\lambda)$, so
$\sqrt{K\cdot \sum_\lambda J_\lambda}=\sqrt{\sum_\lambda KJ_\lambda}
= \sqrt{\sum_\lambda \sqrt{KJ_\lambda}}$.
Finally, again by \Cref{lem:ring-radical-product} (applied to $K$ and
$J_\lambda$), $\sqrt{KJ_\lambda}=K\cap J_\lambda$ for all $\lambda$, hence
$\sqrt{\sum_\lambda \sqrt{KJ_\lambda}}=\sqrt{\sum_\lambda (K\cap J_\lambda)}$.
Combining these equalities yields the distributivity law, so $\RadId(R)$ is a
frame.
\end{proof}

\begin{Thm}\label{thm:ring-points-primes}
There is a bijection $\pt(\RadId(R))\cong\Spec(R)$, under which
$U(I)$ corresponds to $ D(I)=\{\fp\in\Spec(R)\mid I\not\subseteq \fp\}$.
In particular, this bijection is a homeomorphism.
\end{Thm}

\begin{proof}
By \Cref{rem:points-meet-prime}, points of $\RadId(R)$ correspond to meet-prime
elements, i.e., radical ideals $P\ne R$ such that
\begin{equation}\label{eq:ring-meet-prime}
I\cap J\subseteq P \Longrightarrow I\subseteq P\ \text{or}\ J\subseteq P
\end{equation}
for all radical ideals $I,J$.

It is well known that \eqref{eq:ring-meet-prime} is equivalent to $P$ being a
prime ideal. Indeed, if $P$ is prime and $I\cap J\subseteq P$, then
$IJ\subseteq I\cap J\subseteq P$, hence $I\subseteq P$ or $J\subseteq P$.
Conversely, suppose $P$ satisfies \eqref{eq:ring-meet-prime} and let $fg\in P$.
Then $\sqrt{(fg)}\subseteq P$, and since
$\sqrt{(fg)}=\sqrt{(f)}\cap\sqrt{(g)}$ we obtain $\sqrt{(f)}\subseteq P$ or
$\sqrt{(g)}\subseteq P$, i.e.\ $f\in P$ or $g\in P$.

Finally, by definition of the topology on $\pt(\RadId(R))$, the basic open
$U(I)$ consists of those points $p$ with $p(I)=1$, i.e.\ those meet-prime
elements $P$ with $I\not\subseteq P$. Under the identification
$P\leftrightarrow \fp$, this is exactly the Zariski basic open $D(I)$. Since the
sets $U(I)$ form a basis, the bijection is a homeomorphism.
\end{proof}

\begin{Thm}\label{thm:ring-spatial}
The frame $\RadId(R)$ is spatial. In particular, there is a canonical isomorphism
of frames $\RadId(R)\cong\Omega(\Spec(R))$.
\end{Thm}

\begin{proof}
By \Cref{prop:spatial-separation}, it suffices to show that if $I\not\subseteq J$
for radical ideals $I,J$, there is a prime $P$ with $I\not\subseteq P$ and
$J\subseteq P$. Choose $f\in I\setminus J$. Since $J=\sqrt{J}=\bigcap_{P\supseteq J}P$ is the intersection of
the primes containing it, some prime $P\supseteq J$ has $f\notin P$.
\end{proof}

\begin{Thm}\label{thm:ring-coherent}
The frame $\RadId(R)$ is coherent, with compact elements the radical finitely
generated ideals $\RadFg(R)$.
\end{Thm}

\begin{proof}
Every radical ideal is the join of the principal radical ideals it contains, so
$\RadFg(R)$ generates $\RadId(R)$ under joins. For finite meets:
$\sqrt{(f_1,\dots,f_m)}\cap\sqrt{(g_1,\dots,g_n)}=\sqrt{(f_ig_j : i,j)}$,
which is radical finitely generated. For compactness, if
$K=\sqrt{(f_1,\dots,f_n)}\subseteq\sqrt{\sum_\lambda I_\lambda}$, then for each
$i$ there exists $N_i\ge 1$ such that $f_i^{N_i}\in\sum_\lambda I_\lambda$.
Hence for each $i$ there exists a finite subset $\Lambda_i$ such that
$f_i^{N_i}\in\sum_{\lambda\in\Lambda_i} I_\lambda$. Taking the union of these
finite sets over $i=1,\dots,n$, we obtain a finite subset $\Lambda_0$ such that
$f_i\in\sqrt{\sum_{\lambda\in\Lambda_0} I_\lambda}$ for all $i$, hence
$K\subseteq\sqrt{\sum_{\lambda\in\Lambda_0} I_\lambda}$. Thus every element of
$\RadFg(R)$ is compact. In particular, the top element $R=\sqrt{(1)}$ is
compact. Hence $\RadId(R)$ is coherent.

Conversely, let $I\in\RadId(R)$ be compact. Since $I$ is radical,
\[
I=\sqrt{\sum_{f\in I}(f)}=\bigvee_{f\in I}\sqrt{(f)}.
\]
By compactness, there exist $f_1,\dots,f_n\in I$ such that
\[
I\subseteq\bigvee_{j=1}^n\sqrt{(f_j)}.
\]
Since each $\sqrt{(f_j)}\subseteq I$, the reverse inclusion is automatic, so
\[
I=\bigvee_{j=1}^n\sqrt{(f_j)}=\sqrt{(f_1,\dots,f_n)}.
\]
Hence $I\in\RadFg(R)$, and the compact elements of $\RadId(R)$ are exactly the
radical finitely generated ideals.
\end{proof}

\begin{Cor}\label{cor:ring-spectral}
For any ring $R$, the spectrum $\Spec(R)$ is a spectral space.
\end{Cor}

\begin{proof}
This follows by combining
\Cref{prop:coherent-spectral,thm:ring-points-primes,thm:ring-coherent}.
\end{proof}

\section{Tambara functors and their ideals}\label{sec:tambara}

Fix a finite group $G$ throughout. We introduce Tambara functors and their ideal
theory. For background on Mackey functors see \cite{ThevenazWebb1995Structure};
for Tambara functors see
\cite{Tambara1993Multiplicative,strickland2012tambarafunctors,HillMazur2019Equivariant,chan2025tambaraaffineline}.

\begin{Def}\label{def:tambara-functor}
A \emph{$G$-Tambara functor} $T$ consists of:
\begin{enumerate}
\item for each subgroup $H\le G$, a commutative ring $T(G/H)$;
\item for each inclusion $K\le H$, a ring homomorphism
$r^H_K\colon T(G/H)\to T(G/K)$ called \emph{restriction};
\item for each inclusion $K\le H$, an additive map
$t^H_K\colon T(G/K)\to T(G/H)$ called \emph{transfer};
\item for each inclusion $K\le H$, a multiplicative map
$N^H_K\colon T(G/K)\to T(G/H)$ called \emph{norm};
\item for each $g\in G$ and subgroup $H\le G$, a ring isomorphism
$c_{g,H}\colon T(G/H)\xrightarrow{\cong} T(G/gHg^{-1})$ called
\emph{conjugation}.
\end{enumerate}
These data are required to satisfy the following axioms.
\begin{enumerate}
\item[{\textup{(T1)}}] (Functoriality)
For subgroups $L\le K\le H$ and $g,h\in G$,
\[
r^K_L\circ r^H_K=r^H_L,\qquad
N^H_K\circ N^K_L=N^H_L,\qquad
t^H_K\circ t^K_L=t^H_L,
\]
and $c_{g,hHh^{-1}}\circ c_{h,H}=c_{gh,H}$.

\item[\textup{(T2)}] (Conjugacy compatibility)
For $K\le H$ and $g\in G$,
\[
c_{g,K}\circ r^H_K = r^{gHg^{-1}}_{gKg^{-1}}\circ c_{g,H},
\qquad
c_{g,H}\circ N^H_K = N^{gHg^{-1}}_{gKg^{-1}}\circ c_{g,K},
\qquad
c_{g,H}\circ t^H_K = t^{gHg^{-1}}_{gKg^{-1}}\circ c_{g,K}.
\]

\item[\textup{(T3)}] (Additive double coset formula)
For subgroups $H,L\le K$,
\[
r^K_L\circ t^K_H
=
\sum_{[\gamma]\in L\backslash K/H}
t^L_{L\cap \gamma H\gamma^{-1}}
\circ r^{\gamma H\gamma^{-1}}_{L\cap \gamma H\gamma^{-1}}
\circ c_{\gamma,H}.
\]

\item[\textup{(T4)}] (Multiplicative double coset formula)
For subgroups $H,L\le K$ and $a\in T(G/H)$,
\[
r^K_L\bigl(N^K_H(a)\bigr)
=
\prod_{[\gamma]\in L\backslash K/H}
N^L_{L\cap \gamma H\gamma^{-1}}
\Bigl(
r^{\gamma H\gamma^{-1}}_{L\cap \gamma H\gamma^{-1}}
\bigl(c_{\gamma,H}(a)\bigr)
\Bigr).
\]

\item[\textup{(T5)}] (Frobenius reciprocity)
For $K\le H$, $x\in T(G/H)$, and $y\in T(G/K)$,
\[
t^H_K\bigl(r^H_K(x)\cdot y\bigr)=x\cdot t^H_K(y).
\]

\item[\textup{(T6)}] (Tambara reciprocity) The norm maps satisfy Tambara
reciprocity for sums and for transfers, as in
\cite[Theorems~2.5 and~2.7]{HillMazur2019Equivariant}; see also
\cite[Proposition~2.15(f)]{CalleChanMehrleQuigleySpitzNiel2026Spectrum}.
\end{enumerate}
When $G=1$, a $G$-Tambara functor is an ordinary commutative ring.
\end{Def}

\begin{Rem}
The axioms involving only $r,t,$ and $c$ recover the underlying Mackey functor,
and (T5) upgrades this to a Green functor. The genuinely Tambara structure is
carried by the norm maps together with the additional axioms (T4) and (T6). We
intentionally leave (T6) somewhat implicit, since Tambara reciprocity will not
play a major explicit role below.
\end{Rem}

\begin{Rem}
There are many equivalent ways to define a Tambara functor; we have followed
\cite[Definition~2.11]{HillMazur2019Equivariant}. For example, one may define a
Tambara functor as a product-preserving functor $\underline{T}$ from the category
of bispans of finite $G$-sets to sets, together with the property that each
$\underline{T}(U)$ is a commutative ring; see
\cite[Section~8]{Tambara1993Multiplicative}. The description above then comes
from the observation that every finite $G$-set decomposes as a coproduct of
transitive $G$-sets.
\end{Rem}

\begin{Exa}\label{exa:C2-reciprocity}
Let $S$ be a $C_2$--Tambara functor, which we can represent as follows:
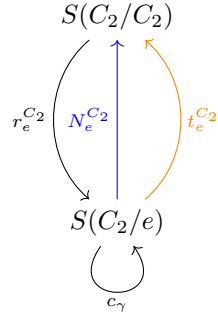
\[
\begin{tikzcd}
S(C_2/C_2)
    \ar[d, bend right=50,"{r^{C_2}_e}"']
    \\[1.5cm]
S(C_2/e)
    \ar[u, bend right=50,orange, "t_{e}^{C_2}"']
    \ar[u, blue, "{N}_e^{C_2}"]
    \arrow["{c_{\gamma}}"', from=2-1, to=2-1, loop, in=305, out=235, distance=10mm]
\end{tikzcd}
\]
Let $a,b\in S(C_2/e)$. Then Tambara reciprocity for sums takes the form
\[
N^{C_2}_e(a+b)=N^{C_2}_e(a)+N^{C_2}_e(b)+t^{C_2}_e(a\cdot c_\gamma(b)).
\]
This is a special case of \cite[Corollary~2.6]{HillMazur2019Equivariant}.
Tambara reciprocity for transfers is trivial in this case because there is no proper
chain of subgroups $H'<H<C_2$, so no additional nontrivial formula arises.
\end{Exa}

\begin{Exa}\label{exa:burnside}
The \emph{Burnside Tambara functor} $A_G$ has $A_G(G/H)=A(H)$, the Burnside
ring of $H$, with the standard restriction, transfer, and norm maps. It is the
initial $G$-Tambara functor. In the case $G = C_2$, this looks as follows,
where the $C_2$-action on $\mathbb{Z}$ is trivial:
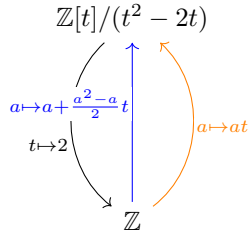
\[
\begin{tikzcd}
\mathbb{Z}[t]/(t^2-2t)
    \ar[d, bend right=50,
        "t \mapsto 2"'{yshift=-7pt, fill=white, inner sep=1pt}]
    \\[1.5cm]
\mathbb{Z}
    \ar[u, bend right=50, orange,
        "a \mapsto at"'{fill=white, inner sep=1pt}]
    \ar[u, blue,
        "a \mapsto a +\frac{a^2-a}{2}t"{yshift=7pt, fill=white, inner sep=1pt}]
\end{tikzcd}
\]
One easily verifies that the Tambara reciprocity formula in
\Cref{exa:C2-reciprocity} holds.
\end{Exa}

\begin{Exa}\label{exa:constant-tambara}
For any commutative ring $R$, the \emph{constant} Tambara functor
$\underline{R}$ has $\underline{R}(G/H)=R$ for all $H$, with
$r^H_K=\mathrm{id}$, $t^H_K(x)=[H:K]\cdot x$, $N^H_K(x)=x^{[H:K]}$, and
$c_g=\mathrm{id}$. In the case $G = C_2$ and $R = \bbZ$, this looks as follows:
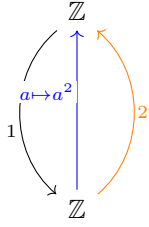
\[
\begin{tikzcd}
\mathbb{Z}
    \ar[d, bend right=50,
        "1"'{yshift=-7pt, fill=white, inner sep=1pt}]
    \\[1.5cm]
\mathbb{Z}
    \ar[u, bend right=50, orange,
        "2"'{fill=white, inner sep=1pt}]
    \ar[u, blue,
        "a \mapsto a^2"{yshift=7pt, fill=white, inner sep=1pt}]
\end{tikzcd}
\]
In this case, the Tambara reciprocity formula is simply the identity
$(a+b)^2=a^2+b^2+2ab$.
\end{Exa}

\begin{Def}\label{def:tambara-ideal}
A \emph{(Tambara) ideal} $I\trianglelefteq T$ is a collection of additive
subgroups $I(G/H)\subseteq T(G/H)$, one for each $H\le G$, satisfying:
\begin{enumerate}
\item[\textup{(I1)}] Each $I(G/H)$ is an ideal of the ring $T(G/H)$.
\item[\textup{(I2)}] $r^H_K\bigl(I(G/H)\bigr)\subseteq I(G/K)$ for all
$K\le H$.
\item[\textup{(I3)}] $t^H_K\bigl(I(G/K)\bigr)\subseteq I(G/H)$ for all
$K\le H$.
\item[\textup{(I4)}] $N^H_K\bigl(I(G/K)\bigr)\subseteq I(G/H)$ for all
$K\le H$.
\item[\textup{(I5)}] $c_g\bigl(I(G/H)\bigr)= I(G/gHg^{-1})$ for all $g\in G$,
$H\le G$.
\end{enumerate}
The ideal $I$ is \emph{proper} if $1\notin I(G/G)$.
\end{Def}

\begin{Rem}
As noted in \cite{wisdom:2508.09360v1} (where it is credited to Natalie Stewart
and David Chan), the condition $N^H_K(I(G/K)) \subseteq N^H_K(0)+I(G/H)$ given
by Nakaoka is not needed when giving a levelwise definition of ideal (as opposed
to indexing on all finite $G$-sets).
\end{Rem}

\begin{Rem}\label{rem:generated-ideal}
By \cite[Proposition~3.2]{Nakaoka2012Ideals} (or directly from the definitions),
the levelwise intersection of Tambara ideals is again a Tambara ideal.

If
 $\{I_\lambda\}_{\lambda\in \Lambda}$ is a family of Tambara ideals, then by
\cite[Proposition~3.13]{Nakaoka2012Ideals} the levelwise sum
\[
\Big(\sum_{\lambda\in \Lambda} I_\lambda\Big)(G/H)
\coloneqq \sum_{\lambda\in \Lambda} I_\lambda(G/H)
\]
is again a Tambara ideal, namely the smallest Tambara ideal containing each
$I_\lambda$. If
$X=\{x_\alpha\}_{\alpha\in A}$ is a collection of elements with
$x_\alpha\in T(G/H_\alpha)$, then $\langle X\rangle$ denotes the Tambara ideal
generated by $X$, i.e., the smallest Tambara ideal containing each $x_\alpha$ at level
$G/H_\alpha$, or equivalently the intersection of all Tambara ideals containing
each $x_\alpha$. For $x\in T(G/H)$, we write $\langle x \rangle$ (or
$\langle x\rangle_H$ if we wish to make the subgroup $H$ clear) for the smallest
Tambara ideal containing $x$ in $T(G/H)$.
\end{Rem}

\begin{Def}\label{def:fg-tambara-ideal}
A Tambara ideal $I\trianglelefteq T$ is \emph{finitely generated} if there exist
finitely many subgroups $H_1,\dots,H_n\le G$ (not necessarily distinct) and
elements $x_i\in T(G/H_i)$ such that $I=\langle x_1,\dots,x_n\rangle$. We write
$\RadFg_G(T)$ for the set of radical finitely generated Tambara ideals, ordered
by inclusion.
\end{Def}

\begin{Def}\label{def:tambara-product}
For ideals $I,J\trianglelefteq T$, the \emph{product} $IJ$ is the Tambara ideal
generated by the levelwise products $\{I(G/H)\cdot J(G/H)\}_{H\le G}$.
\end{Def}

\begin{Rem}\label{rem:product-not-levelwise}
In contrast to the ring case, the product $IJ$ is generally \emph{not} computed
levelwise. Indeed, $(IJ)(G/H)$ may be strictly larger than $I(G/H)\cdot J(G/H)$, since
the Tambara ideal generated by the levelwise products must also be closed under
transfers and norms.
\end{Rem}

The following useful terminology comes from \cite{chan2025tambaraaffineline}.

\begin{Def}\label{def:mult-translate}
Let $x\in T(G/H)$. A \emph{multiplicative translate} of $x$ into $T(G/L)$ is
any element of the form
$N^L_{gKg^{-1}}\bigl(c_{g,K}(r^H_K(x))\bigr)$
for some $K\le H$, $g\in G$ with $gKg^{-1}\le L$. A \emph{generalized product}
of $x\in T(G/H_1)$ and $y\in T(G/H_2)$ is a product $\mu\cdot\nu\in T(G/L)$
where $\mu$ is a multiplicative translate of $x$ into $T(G/L)$ and $\nu$ is a
multiplicative translate of $y$ into $T(G/L)$, for some $L\le G$.
\end{Def}

\begin{Rem}\label{rem:finitely-many-translates}
Since $G$ is finite, for any $x\in T(G/H)$ and any $L\le G$ there are only
finitely many multiplicative translates of $x$ into $T(G/L)$, since they arise
from finitely many choices of $K\le H$ and $g\in G$ with $gKg^{-1}\le L$.
Consequently, for fixed $x$ and $y$, the set of generalized products of $x$ and
$y$ is finite.
\end{Rem}

The following definition first appeared in
\cite[Definition~2.6]{CalleGinnett2023Spectrum}.

\begin{Def}\label{def:Q}
Let $I\trianglelefteq T$ be a Tambara ideal, and let $x\in T(G/H_1)$,
$y\in T(G/H_2)$. We write $\mathcal{Q}(I,x,y)$ for the statement that every
generalized product of $x$ and $y$ lies in $I$.
\end{Def}

\begin{Rem}\label{rem:Q-analogue}
The condition $\mathcal{Q}(I,x,y)$ is the Tambara-theoretic analogue of the
ring-theoretic condition $xy\in I$.
\end{Rem}

\begin{Prop}[{\cite[Corollary~3.12]{Nakaoka2012Ideals}}]\label{prop:product-criterion-Q}
Let $I\trianglelefteq T$ be a Tambara ideal, and let $a\in T(G/H_1)$,
$b\in T(G/H_2)$. Then the following are equivalent:
\begin{enumerate}
\item[\textup{(1)}] $\langle a\rangle\langle b\rangle\subseteq I$.
\item[\textup{(2)}] $\mathcal{Q}(I,a,b)$ holds, i.e., every generalized product
of $a$ and $b$ lies in $I$.
\end{enumerate}
\end{Prop}

\begin{Cor}\label{cor:product-generators}
Let $a\in T(G/H_1)$ and $b\in T(G/H_2)$. The product
$\langle a\rangle\langle b\rangle$ is generated as a Tambara ideal by the set of
all generalized products of $a$ and $b$.
\end{Cor}

\begin{proof}
Let $J$ denote the Tambara ideal generated by all generalized products of $a$
and $b$. By definition, $\mathcal{Q}(J,a,b)$ holds, so
$\langle a\rangle\langle b\rangle\subseteq J$ by
\Cref{prop:product-criterion-Q}.

For the reverse inclusion, let $z=\mu\cdot\nu\in T(G/L)$ be a generalized
product of $a$ and $b$, where $\mu$ is a multiplicative translate of $a$ and
$\nu$ is a multiplicative translate of $b$. Since $\langle a\rangle$ is a
Tambara ideal containing $a$, it contains every multiplicative translate of $a$,
so $\mu\in \langle a\rangle(G/L)$. Similarly, $\nu\in \langle b\rangle(G/L)$.
Therefore
$z=\mu\cdot\nu\in \langle a\rangle(G/L)\cdot \langle b\rangle(G/L)\subseteq
(\langle a\rangle\langle b\rangle)(G/L)$.
Thus every generalized product lies in $\langle a\rangle\langle b\rangle$, so
$J\subseteq \langle a\rangle\langle b\rangle$, and hence
$\langle a\rangle\langle b\rangle = J$.
\end{proof}

\begin{Cor}\label{cor:fg-product}
The product of two finitely generated Tambara ideals is finitely generated.
\end{Cor}

\begin{proof}
Let $I=\langle x_1,\dots,x_m\rangle$ and $J=\langle y_1,\dots,y_n\rangle$.
Then $I=\sum_{i=1}^m \langle x_i\rangle$ and
$J=\sum_{j=1}^n \langle y_j\rangle$. Since Tambara products are defined as the
Tambara ideal generated by the levelwise ring products, and ring multiplication
distributes over finite sums, both sides of
$IJ=\sum_{i,j}\langle x_i\rangle\langle y_j\rangle$
are the Tambara ideal generated by
$\sum_{i,j}\langle x_i\rangle(G/H)\cdot \langle y_j\rangle(G/H)$ at each level
$G/H$. By \Cref{cor:product-generators}, each
$\langle x_i\rangle\langle y_j\rangle$ is generated by the finitely many
generalized products of $x_i$ and $y_j$ (\Cref{rem:finitely-many-translates}),
hence is finitely generated. Therefore $IJ$ is a finite sum of finitely
generated Tambara ideals, and is itself finitely generated.
\end{proof}

We now move on to prime ideals.

\begin{Def}\label{def:tambara-prime}
A proper ideal $P\trianglelefteq T$ is \emph{prime} if for all ideals
$I,J\trianglelefteq T$,
\[
IJ\subseteq P \Longrightarrow I\subseteq P \text{ or } J\subseteq P.
\]
We write $\Spec_{\Nak}(T)$ for the set of prime ideals of $T$.
\end{Def}

\begin{Rem}\label{rem:levels-not-prime}
Unlike for Tambara ideals, the levels $P(G/H)$ of a Nakaoka-prime ideal $P$ are
\emph{not} necessarily prime ideals of the rings $T(G/H)$; they are, however,
radical ideals \cite[Theorem~4.7]{chan2025tambaraaffineline}.
\end{Rem}

Prime ideals can also be characterised in terms of generalized products.

\begin{Thm}[{\cite[Proposition~4.2]{Nakaoka2014Spectrum}}]\label{thm:Q-characterization-prime}
A proper ideal $P\trianglelefteq T$ is prime if and only if for any $a\in T(G/H)$,
$b\in T(G/H')$, the statement $\mathcal{Q}(P,a,b)$ implies $a\in P(G/H)$ or
$b\in P(G/H')$.
\end{Thm}

\begin{Def}\label{def:nakaoka-topology}
The \emph{Nakaoka spectrum} $\Spec_{\Nak}(T)$ is topologised by declaring the
sets
$V_H(x)\coloneqq \{P\in \Spec_{\Nak}(T)\mid x\in P(G/H)\}$,
for $H\le G$ and $x\in T(G/H)$,
to be a subbasis for the closed sets. The complementary subbasic opens are
$D_H(x)\coloneqq \Spec_{\Nak}(T)\setminus V_H(x)$. When $G=1$ this recovers
the classical Zariski topology.
\end{Def}

\begin{Rem}\label{rem:D-basis}
Although the sets $D_H(x)$ are defined as subbasic opens, they in fact form a
basis for the Nakaoka topology. Indeed, let
$P\in D_H(x)\cap D_{H'}(y)$. Then $x\notin P(G/H)$ and $y\notin P(G/H')$. Since
$P$ is prime, the contrapositive of \Cref{thm:Q-characterization-prime} shows
that $\mathcal{Q}(P,x,y)$ fails. Hence there exists a generalized product
$z\in T(G/L)$ of $x$ and $y$ with $z\notin P(G/L)$, so $P\in D_L(z)$.

Moreover, $D_L(z)\subseteq D_H(x)\cap D_{H'}(y)$. For if $Q$ is prime and
$x\in Q(G/H)$, then $Q$, being a Tambara ideal, contains every multiplicative
translate of $x$, hence every generalized product of $x$ and $y$, in particular
$z$. Thus $z\notin Q(G/L)$ implies $x\notin Q(G/H)$, and similarly
$z\notin Q(G/L)$ implies $y\notin Q(G/H')$.
\end{Rem}

\begin{Def}[Nakaoka]\label{def:tambara-radical}
Let $I\trianglelefteq T$ be a Tambara ideal. The \emph{radical} of $I$, denoted
$\sqrt{I}$, is the Tambara ideal defined by
\[
(\sqrt{I})(G/H)\coloneqq
\{\,x\in T(G/H)\mid \exists\, n\ge 1\ \text{such that}\
\langle x\rangle^{n}\subseteq I\,\},
\]
where $\langle x\rangle\trianglelefteq T$ denotes the Tambara ideal generated by
$x$ and $\langle x\rangle^{n}$ is its $n$-fold ideal product. An ideal $I$ is
\emph{radical} if $I=\sqrt{I}$.
\end{Def}

The following is \cite[Corollary~4.7]{ChanSpitz2026RadicalsNilpotents}.

\begin{Prop}[Chan--Spitz]\label{prop:levelwise-radical}
For any Tambara ideal $I\trianglelefteq T$ and subgroup $H\le G$, the ring ideal
$(\sqrt{I})(G/H)\subseteq T(G/H)$ is the ordinary radical of $I(G/H)$. Moreover,
\[
\sqrt{I}=\bigcap_{\substack{P\in\Spec_{\Nak}(T)\\ I\subseteq P}} P.
\]
In particular, $\sqrt{I}\trianglelefteq T$ is a radical Tambara ideal containing
$I$.
\end{Prop}
We now establish a Tambara analog of \Cref{lem:ring-radical-product}. We start with the following. 
\begin{Lem}\label{lem:radical-product-levelwise}
    For any Tambara ideals $I, J$ and any  subgroup $H \le G$,
    \[
        \sqrt{(IJ)(G/H)} = \sqrt{I(G/H)} \cap \sqrt{J(G/H)}.
    \]
\end{Lem}
\begin{proof}
Note that there are inclusions
    \[
        I(G/H) \cdot J(G/H) \subseteq (IJ)(G/H) \subseteq I(G/H) \cap J(G/H),
    \]
    where the first follows from the definition of the Tambara product
    (\Cref{def:tambara-product}) and the second from $IJ \subseteq I$ and
    $IJ \subseteq J$. Taking radicals, we get
    \[
        \sqrt{I(G/H) \cdot J(G/H)}
        \subseteq \sqrt{(IJ)(G/H)}
        \subseteq \sqrt{I(G/H) \cap J(G/H)}.
    \]
    By \Cref{lem:ring-radical-product} applied in the ring $T(G/H)$, the
    outer two terms are both equal to $\sqrt{I(G/H)} \cap \sqrt{J(G/H)}$,
    and the claim follows.
\end{proof}

\begin{Prop}\label{prop:radical-product-intersection-general}
    For any Tambara ideals $I, J$,
    \[
        \sqrt{IJ} = \sqrt{I} \cap \sqrt{J}.
    \]
\end{Prop}
\begin{proof}
    We argue levelwise. For any subgroup $H \le G$,
    \begin{align*}
        (\sqrt{IJ})(G/H)
        &= \sqrt{(IJ)(G/H)} \\
        &= \sqrt{I(G/H)} \cap \sqrt{J(G/H)} \\
        &= (\sqrt{I})(G/H) \cap (\sqrt{J})(G/H) \\
        &= (\sqrt{I} \cap \sqrt{J})(G/H),
    \end{align*}
    where the first and third equalities follow from \Cref{prop:levelwise-radical},
    the second is \Cref{lem:radical-product-levelwise}, and the fourth is
    the definition of the intersection of Tambara ideals. Since $H \le G$
    was arbitrary, $\sqrt{IJ} = \sqrt{I} \cap \sqrt{J}$.
\end{proof}

\section{The frame of radical Tambara ideals}\label{sec:zariski-tambara}

Fix a finite group $G$ and a $G$-Tambara functor $T$. In this section, we construct the point-free Nakaoka spectrum. The arguments parallel those of
\Cref{sec:ring-zariski}, with the identity $\sqrt{IJ}=\sqrt{I}\cap\sqrt{J}$
(\Cref{prop:radical-product-intersection-general}) playing the same key
role as its ring-theoretic counterpart (\Cref{lem:ring-radical-product}).

\begin{Not}\label{not:radid-tambara}
Let $\RadId_G(T)$ denote the set of radical Tambara ideals of $T$, ordered by
inclusion.
\end{Not}

\begin{Prop}\label{prop:radid-frame}
$\RadId_G(T)$ is a frame. Meets are given by levelwise intersections, and joins
are given by
$\bigvee_{\lambda} I_\lambda = \sqrt{\sum_{\lambda} I_\lambda}$,
where $\sum_{\lambda} I_\lambda$ denotes the Tambara ideal with
$(\sum_{\lambda} I_\lambda)(G/H)=\sum_{\lambda} I_\lambda(G/H)$.
\end{Prop}

\begin{proof}
Intersections of radical Tambara ideals are radical, so $\RadId_G(T)$ has
arbitrary meets.

Let $\{I_\lambda\}_\lambda$ be a family of radical Tambara ideals and set
$J=\sum_{\lambda} I_\lambda$. By \Cref{rem:generated-ideal}, this levelwise sum
is a Tambara ideal. Its radical $\sqrt{J}$ (in the sense of
\Cref{def:tambara-radical}) is a
radical Tambara ideal containing each $I_\lambda$, hence it is an upper bound.
If $K\in \RadId_G(T)$ is any other upper bound, then $J\subseteq K$, so
$\sqrt{J}\subseteq \sqrt{K}=K$ by monotonicity and idempotence of
$\sqrt{(-)}$. Thus $\sqrt{J}$ is the join.

Finally, we verify frame distributivity. Fix $K\in \RadId_G(T)$ and a family
$\{J_\lambda\}_\lambda$ in $\RadId_G(T)$. By \Cref{prop:levelwise-radical}
applied to $\sum_\lambda J_\lambda$, the join $\sqrt{\sum_\lambda J_\lambda}$
is computed levelwise:
$(\sqrt{\sum_\lambda J_\lambda})(G/H)=\sqrt{\sum_\lambda J_\lambda(G/H)}$ in
$T(G/H)$. Since intersections of Tambara ideals are also computed levelwise,
it suffices to check the distributivity identity in each ring $T(G/H)$:
\[
K(G/H)\cap \sqrt{\sum_{\lambda} J_\lambda(G/H)}
=
\sqrt{\sum_{\lambda}\bigl(K(G/H)\cap J_\lambda(G/H)\bigr)}.
\]
This holds because $\RadId(T(G/H))$ is a frame (\Cref{prop:radid-ring-frame}).
\end{proof}

\begin{Def}\label{def:specfrm}
The \emph{frame-theoretic Nakaoka spectrum} of $T$ is
$\Spec_{\Frm}(T)\coloneqq\pt(\RadId_G(T))$, equipped with the topology of
\Cref{def:pt-frame}.
\end{Def}

\begin{Thm}\label{thm:points-primes}
There is a canonical homeomorphism
$\Spec_{\Frm}(T)\cong\Spec_{\Nak}(T)$,
identifying the basic open $U(I)=\{p\mid p(I)=1\}$ with the set
$\{P\in\Spec_{\Nak}(T)\mid I\not\subseteq P\}$. 
\end{Thm}

\begin{proof}
By \Cref{rem:points-meet-prime}, points of $\RadId_G(T)$ correspond to
meet-prime elements, i.e., radical Tambara ideals $P\ne T$ such that
\begin{equation}\label{eq:meet-prime-radid}
I\cap J\subseteq P \Longrightarrow I\subseteq P\ \text{or}\ J\subseteq P
\end{equation}
for all $I,J\in\RadId_G(T)$. We show this is equivalent to $P$ being
Nakaoka-prime (\Cref{def:tambara-prime}).

Firstly, let $P$ be Nakaoka-prime and suppose $I\cap J\subseteq P$ for radical Tambara ideals $I,J$. Since $IJ\subseteq I\cap J\subseteq P$ and $P$ is Nakaoka-prime, $I\subseteq P$ or $J\subseteq P$.

Conversely, suppose that $P$ satisfies \eqref{eq:meet-prime-radid}. Since
$P\ne T$, it is proper. By
\Cref{thm:Q-characterization-prime}, it suffices to show that for all
$a\in T(G/H_1)$ and $b\in T(G/H_2)$,
$\mathcal{Q}(P,a,b)$ implies $a\in P(G/H_1)$ or $b\in P(G/H_2)$.
So assume $\mathcal{Q}(P,a,b)$. By \Cref{cor:product-generators}, the ideal
$\langle a\rangle\langle b\rangle$ is generated by the generalized products of
$a$ and $b$, hence $\langle a\rangle\langle b\rangle\subseteq P$. By
\Cref{prop:radical-product-intersection-general},
\[
\sqrt{\langle a\rangle}\cap \sqrt{\langle b\rangle}
= \sqrt{\langle a\rangle\langle b\rangle}
\subseteq P.
\]
Since $P$ is meet-prime, either $\sqrt{\langle a\rangle}\subseteq P$ or
$\sqrt{\langle b\rangle}\subseteq P$. As
$a\in \sqrt{\langle a\rangle}(G/H_1)$ and
$b\in \sqrt{\langle b\rangle}(G/H_2)$, it follows that $a\in P(G/H_1)$ or
$b\in P(G/H_2)$.

\smallskip\noindent
For the topology, note that $U(I)$ in $\pt(\RadId_G(T))$ corresponds to
$\{P\in\Spec_{\Nak}(T)\mid I\not\subseteq P\}$. The Nakaoka closed set
$V_H(x)$ corresponds to $\{P\mid \sqrt{\langle x\rangle_H}\subseteq P\}$,
since $x\in P(G/H)$ if and only if $\sqrt{\langle x\rangle_H}\subseteq P$ (because $P$ is
radical). Equivalently,
\[
D_H(x)=\{P\mid x\notin P(G/H)\}=U\bigl(\sqrt{\langle x\rangle_H}\bigr),
\]
so every Nakaoka basic open is open in $\Spec_{\Frm}(T)$. Conversely, if
$P\in U(I)$, then $I\not\subseteq P$, so there exist $H\le G$ and
$x\in I(G/H)\setminus P(G/H)$. Then $P\in D_H(x)\subseteq U(I)$, since any
prime $Q$ with $x\notin Q(G/H)$ cannot contain $I$. Hence
\[
U(I)=\bigcup_{H\le G}\bigcup_{x\in I(G/H)} D_H(x),
\]
so every basic open of $\Spec_{\Frm}(T)$ is open in the Nakaoka topology. Thus
the two topologies agree.
\end{proof}

\begin{Thm}\label{thm:spatiality}
The frame $\RadId_G(T)$ is spatial.
\end{Thm}

\begin{proof}
By \Cref{prop:spatial-separation}, it suffices to show that if $I\not\subseteq J$
for radical Tambara ideals $I,J$, there exists a Nakaoka prime $P$ with
$I\not\subseteq P$ and $J\subseteq P$. Since $I\not\subseteq J$, there exist
$H\le G$ and $x\in I(G/H)\setminus J(G/H)$. By \Cref{prop:levelwise-radical}
applied to $J$ (which is radical, so $J=\sqrt{J}$), we have
$J=\bigcap_{P\supseteq J}P$. In particular, there exists some prime containing
$J$ as otherwise the right-hand side would be $T$, contradicting $x\notin J(G/H)$.
Since intersections are computed levelwise,
$J(G/H)=\bigcap_{P\supseteq J}P(G/H)$. Hence some prime $P\supseteq J$
satisfies $x\notin P(G/H)$, and consequently $I\not\subseteq P$.
\end{proof}

\begin{Cor}\label{cor:frame-opens}
There is an isomorphism of frames
$\RadId_G(T)\cong\Omega(\Spec_{\Nak}(T))$. 
\end{Cor}

\begin{proof}\sloppy
By \Cref{thm:spatiality}, $\RadId_G(T)\cong\Omega(\Spec_{\Frm}(T))$, and by
\Cref{thm:points-primes}, $\Spec_{\Frm}(T)\cong\Spec_{\Nak}(T)$.
\end{proof}

\section{Spectrality and functoriality}\label{sec:spectrality}

We now show that $\RadId_G(T)$ is coherent, and hence deduce that $\Spec_{\Nak}(T)$ is
spectral. Moreover, we show that a morphism of Tambara functors gives rise to a morphism of \emph{spectral} spaces after applying $\Spec_{\Nak}$. 

For the following, we recall that we write $\RadFg_G(T)$ for the set of radical finitely generated Tambara ideals, ordered by inclusion.

\begin{Lem}\label{lem:radfg-generates}
$\RadFg_G(T)$ generates $\RadId_G(T)$ under joins, i.e., every radical Tambara ideal $I$ satisfies $I=\bigvee_{H\le G}\bigvee_{x\in I(G/H)} \sqrt{\langle x\rangle_H}$.
\end{Lem}

\begin{proof}
Let $I\in \RadId_G(T)$. For each $H\le G$ and $x\in I(G/H)$, the principal
Tambara ideal $\langle x\rangle_H$ is contained in $I$ by minimality of
$\langle x\rangle_H$, hence
$\sqrt{\langle x\rangle_H}\subseteq \sqrt{I}=I$. Therefore
$\bigvee_{H\le G}\bigvee_{x\in I(G/H)} \sqrt{\langle x\rangle_H}
\subseteq I$.

For the reverse inclusion, consider the Tambara ideal
$J\coloneqq \sum_{H\le G}\sum_{x\in I(G/H)} \langle x\rangle_H$.
By construction, $J\subseteq I$ and for each $H$ we have
$I(G/H)\subseteq J(G/H)$, since
$x\in \langle x\rangle_H(G/H)\subseteq J(G/H)$. Hence $J(G/H)=I(G/H)$ for
all $H$, so $J=I$. Taking radicals and using that $I$ is radical gives
$I=\sqrt{I}=\sqrt{J}
=\sqrt{\sum_{H\le G}\sum_{x\in I(G/H)} \langle x\rangle_H}
= \bigvee_{H\le G}\bigvee_{x\in I(G/H)} \sqrt{\langle x\rangle_H}$.
\end{proof}

\begin{Lem}\label{lem:radfg-compact}
Every element of $\RadFg_G(T)$ is compact in $\RadId_G(T)$.
\end{Lem}

\begin{proof}
Let $K\in \RadFg_G(T)$, say $K=\sqrt{\langle x_1,\dots,x_n\rangle}$ with
$x_i\in T(G/H_i)$. Suppose $K\subseteq \bigvee_\lambda I_\lambda
=\sqrt{\sum_\lambda I_\lambda}$. Then each $x_i\in
(\sqrt{\sum_\lambda I_\lambda})(G/H_i)$, which by
\Cref{prop:levelwise-radical} equals the ordinary ring radical
$\sqrt{\sum_\lambda I_\lambda(G/H_i)}$ in $T(G/H_i)$. Equivalently,
$\sqrt{(x_i)}\subseteq\sqrt{\sum_\lambda I_\lambda(G/H_i)}$ in
$\RadId(T(G/H_i))$.

By compactness of $\sqrt{(x_i)}$ in $\RadId(T(G/H_i))$
(\Cref{thm:ring-coherent}), there exists a finite $\Lambda_i$ with
$\sqrt{(x_i)}\subseteq\sqrt{\sum_{\lambda\in\Lambda_i} I_\lambda(G/H_i)}$, i.e.
$x_i\in \sqrt{\sum_{\lambda\in\Lambda_i} I_\lambda(G/H_i)}$. Let
$\Lambda_0=\Lambda_1\cup\cdots\cup\Lambda_n$. Then each
$x_i\in(\sqrt{\sum_{\lambda\in\Lambda_0} I_\lambda})(G/H_i)$ again by
\Cref{prop:levelwise-radical}.

Thus the Tambara ideal $\sqrt{\sum_{\lambda\in\Lambda_0} I_\lambda}$ contains
each $x_i$, hence contains $\langle x_1,\dots,x_n\rangle$, hence contains $K$.
\end{proof}

\begin{Lem}\label{lem:radfg-meets}
$\RadFg_G(T)$ is closed under finite meets in $\RadId_G(T)$.
\end{Lem}

\begin{proof}
Let $I,J\in\RadFg_G(T)$, say $I=\sqrt{\langle x_1,\dots,x_m\rangle}$ and
$J=\sqrt{\langle y_1,\dots,y_n\rangle}$. By
\Cref{prop:radical-product-intersection-general},
\[
I\cap J
= \sqrt{\langle x_1,\dots,x_m\rangle}\cap\sqrt{\langle y_1,\dots,y_n\rangle}
= \sqrt{\langle x_1,\dots,x_m\rangle\cdot\langle y_1,\dots,y_n\rangle}.
\]
By \Cref{cor:fg-product},
$\langle x_1,\dots,x_m\rangle\cdot\langle y_1,\dots,y_n\rangle$ is finitely
generated, so $I\cap J$ is a radical finitely generated Tambara ideal.
\end{proof}

\begin{Thm}\label{thm:coherent}
The frame $\RadId_G(T)$ is coherent, with compact elements the radical finitely
generated Tambara ideals $\RadFg_G(T)$.
\end{Thm}

\begin{proof}
By \Cref{lem:radfg-generates}, $\RadFg_G(T)$ generates $\RadId_G(T)$ under
joins. By \Cref{lem:radfg-compact}, these are compact. By
\Cref{lem:radfg-meets}, they are closed under finite meets. The top element is
$T=\sqrt{\langle 1\rangle_G}$, which is compact by \Cref{lem:radfg-compact}.
Hence $\RadId_G(T)$ is coherent.

Conversely, let $I\in\RadId_G(T)$ be compact. By \Cref{lem:radfg-generates},
\[
I=\bigvee_{H\le G}\bigvee_{x\in I(G/H)} \sqrt{\langle x\rangle_H}.
\]
By compactness, there exist subgroups $H_1,\dots,H_n\le G$ and elements
$x_i\in I(G/H_i)$ such that
\[
I\subseteq \bigvee_{i=1}^n \sqrt{\langle x_i\rangle_{H_i}}.
\]
Since each $\sqrt{\langle x_i\rangle_{H_i}}\subseteq I$, the reverse inclusion
is automatic, so
\[
I=\bigvee_{i=1}^n \sqrt{\langle x_i\rangle_{H_i}}.
\]
Because each $x_i\in I(G/H_i)$, we have
$\langle x_1,\dots,x_n\rangle\subseteq I$, and therefore
\[
\sqrt{\langle x_1,\dots,x_n\rangle}\subseteq I.
\]
On the other hand, each
$\sqrt{\langle x_i\rangle_{H_i}}\subseteq\sqrt{\langle x_1,\dots,x_n\rangle}$,
so
\[
I=\bigvee_{i=1}^n \sqrt{\langle x_i\rangle_{H_i}}
\subseteq \sqrt{\langle x_1,\dots,x_n\rangle}.
\]
Hence
\[
I=\sqrt{\langle x_1,\dots,x_n\rangle}\in\RadFg_G(T),
\]
and the compact elements of $\RadId_G(T)$ are exactly the radical finitely
generated Tambara ideals.
\end{proof}

\begin{Thm}\label{thm:spectrality}
The Nakaoka spectrum $\Spec_{\Nak}(T)$ is a spectral space.
\end{Thm}

\begin{proof}
By \Cref{thm:coherent}, $\RadId_G(T)$ is coherent and by
\Cref{prop:coherent-spectral}, its point space is spectral. But by
\Cref{thm:points-primes}, $\pt(\RadId_G(T))\cong\Spec_{\Nak}(T)$, so the latter is spectral, as claimed. 
\end{proof}

\begin{Rem}\label{rem:compact-opens}
The compact open subsets of $\Spec_{\Nak}(T)$ correspond to the compact elements
of $\RadId_G(T)$, namely the radical finitely generated Tambara ideals. The
compact opens are $D_I=\{P\mid I\not\subseteq P\}$ for $I\in\RadFg_G(T)$.
\end{Rem}
\begin{Rem}
We now turn to functoriality. Nakaoka \cite[Theorem~4.9]{Nakaoka2012Ideals}
shows that for a morphism $\varphi\colon T\to S$ of $G$-Tambara functors and a
prime $Q\in\Spec_{\Nak}(S)$, the levelwise preimage $\varphi^{-1}(Q)$ is a prime
Tambara ideal of $T$, and that the resulting map
$\varphi^a\colon\Spec_{\Nak}(S)\to\Spec_{\Nak}(T)$,
$Q\mapsto\varphi^{-1}(Q)$, is continuous. We will reprove this from the frame theoretic perspective, and in fact show that $\varphi^a$ is a morphism of spectral spaces.
\end{Rem}
\begin{Def}\label{def:extension-radical}
Let $\varphi\colon T\to S$ be a morphism of $G$-Tambara functors. For a Tambara
ideal $I\trianglelefteq T$, let $\varphi_\ast(I)\trianglelefteq S$ denote the
Tambara ideal of $S$ generated by the subsets
$\varphi_H(I(G/H))\subseteq S(G/H)$ for $H\le G$. If $I$ is radical, define
$\widetilde{\varphi}(I)\coloneqq \sqrt{\varphi_\ast(I)}\in \RadId_G(S)$.
\end{Def}

\begin{Lem}\label{lem:preimage-radical}
Let $\varphi\colon T\to S$ be a morphism of $G$-Tambara functors and let
$J\trianglelefteq S$ be a radical Tambara ideal. Then the levelwise preimage
$\varphi^{-1}(J)$, defined by
$\varphi^{-1}(J)(G/H)\coloneqq \varphi_H^{-1}(J(G/H))$, is a radical Tambara
ideal of $T$.
\end{Lem}

\begin{proof}
Closure under the Tambara structure maps follows from the fact that $\varphi$
commutes with restriction, transfer, norm, and conjugation. Radicality holds
levelwise. Indeed, $\varphi_H^{-1}(J(G/H))$ is a radical ideal of $T(G/H)$ since
$J(G/H)$ is a radical ideal of $S(G/H)$ and preimages of radical ideals under
ring homomorphisms are radical.
\end{proof}

\begin{Lem}\label{lem:extension-radical-basic}
Let $\varphi\colon T\to S$ be a morphism of $G$-Tambara functors and let
$I\trianglelefteq T$ be a Tambara ideal. Then
$\sqrt{\varphi_\ast(\sqrt{I})}=\sqrt{\varphi_\ast(I)}$.
\end{Lem}

\begin{proof}
Since $I\subseteq\sqrt{I}$, we have
$\varphi_\ast(I)\subseteq\varphi_\ast(\sqrt{I})$, hence
$\sqrt{\varphi_\ast(I)}\subseteq\sqrt{\varphi_\ast(\sqrt{I})}$. For the
reverse, let $x\in(\sqrt{I})(G/H)$. By \Cref{prop:levelwise-radical},
$x^n\in I(G/H)$ for some $n\ge 1$, so
$\varphi_H(x)^n=\varphi_H(x^n)\in\varphi_H(I(G/H))\subseteq
\varphi_\ast(I)(G/H)$. Again by \Cref{prop:levelwise-radical},
$\varphi_H(x)\in\sqrt{\varphi_\ast(I)}(G/H)$. Thus every generator of
$\varphi_\ast(\sqrt{I})$ lies in $\sqrt{\varphi_\ast(I)}$, so
$\varphi_\ast(\sqrt{I})\subseteq\sqrt{\varphi_\ast(I)}$. Taking radicals gives
the reverse inclusion.
\end{proof}

\begin{Prop}\label{prop:extension-frame-map}
Let $\varphi\colon T\to S$ be a morphism of $G$-Tambara functors. Then
\[
\widetilde{\varphi}\colon \RadId_G(T)\longrightarrow \RadId_G(S),
\qquad
I\longmapsto \sqrt{\varphi_\ast(I)},
\]
is a frame homomorphism.
\end{Prop}

\begin{proof}
We verify that $\widetilde{\varphi}$ preserves the top element, arbitrary joins, and finite meets in turn.

\smallskip\noindent
\emph{Top element.}
$\widetilde{\varphi}(T)=\sqrt{\varphi_\ast(T)}=\sqrt{S}=S$, since
$\varphi_H(T(G/H))$ contains $\varphi_H(1)=1$ for each $H$.

\smallskip\noindent
\emph{Joins.}
Let $\{I_\lambda\}\subseteq\RadId_G(T)$. By \Cref{prop:radid-frame},
$\bigvee_\lambda I_\lambda=\sqrt{\sum_\lambda I_\lambda}$. Using
\Cref{lem:extension-radical-basic},
$\widetilde{\varphi}(\bigvee_\lambda I_\lambda)
=\sqrt{\varphi_\ast(\sqrt{\sum_\lambda I_\lambda})}
=\sqrt{\varphi_\ast(\sum_\lambda I_\lambda)}$.
Since $\varphi_\ast(\sum_\lambda I_\lambda)$ is the Tambara ideal generated by
the union of the subsets $\varphi_H(I_\lambda(G/H))$, we have
$\varphi_\ast(\sum_\lambda I_\lambda)=\sum_\lambda\varphi_\ast(I_\lambda)$, so
$\widetilde{\varphi}(\bigvee_\lambda I_\lambda)
=\sqrt{\sum_\lambda\varphi_\ast(I_\lambda)}
=\bigvee_\lambda\widetilde{\varphi}(I_\lambda)$.

\smallskip\noindent
\emph{Finite meets.}
Let $I,J\in\RadId_G(T)$. We show
$\widetilde{\varphi}(I\cap J)=\widetilde{\varphi}(I)\cap\widetilde{\varphi}(J)$.
By \Cref{prop:radical-product-intersection-general} (in $S$), the right-hand side
equals $\sqrt{\varphi_\ast(I)\cdot\varphi_\ast(J)}$. We claim
\begin{equation}\label{eq:ext-meet}
\sqrt{\varphi_\ast(I\cap J)}=\sqrt{\varphi_\ast(I)\cdot\varphi_\ast(J)}.
\end{equation}

For $\supseteq$ write
\[
\varphi_\ast(I)=\sum_{H\le G}\sum_{x\in I(G/H)}\langle \varphi_H(x)\rangle,
\qquad
\varphi_\ast(J)=\sum_{K\le G}\sum_{y\in J(G/K)}\langle \varphi_K(y)\rangle.
\]
Since sums of Tambara ideals are computed levelwise and Tambara products are
defined as the Tambara ideals generated by the levelwise ring products, we have
\[
\varphi_\ast(I)\cdot\varphi_\ast(J)
=
\sum_{H,K}\sum_{x,y}
\langle\varphi_H(x)\rangle\langle\varphi_K(y)\rangle.
\]
By \Cref{cor:product-generators}, each
$\langle\varphi_H(x)\rangle\langle\varphi_K(y)\rangle$ is generated by the
generalized products of $\varphi_H(x)$ and $\varphi_K(y)$. Since $\varphi$
commutes with all structure maps, every such generalized product is
$\varphi_L(\mu\cdot\nu)$ where $\mu\cdot\nu$ is a generalized product of $x$
and $y$ in $T$. Since $\mu\in I(G/L)$ and $\nu\in J(G/L)$, we have
$\mu\cdot\nu\in(I\cap J)(G/L)$, so
$\varphi_L(\mu\cdot\nu)\in\varphi_\ast(I\cap J)(G/L)$. Thus every generator of
$\varphi_\ast(I)\cdot\varphi_\ast(J)$ lies in $\varphi_\ast(I\cap J)$, giving
$\varphi_\ast(I)\cdot\varphi_\ast(J)\subseteq\varphi_\ast(I\cap J)$.

For $\subseteq$ let $x\in(I\cap J)(G/H)$. Then $x\in I(G/H)\cap J(G/H)$,
so $\varphi_H(x)\in\varphi_H(I(G/H))\cap\varphi_H(J(G/H))\subseteq
\varphi_\ast(I)(G/H)\cap\varphi_\ast(J)(G/H)$. In particular,
$\varphi_H(x)^2=\varphi_H(x)\cdot\varphi_H(x)\in
\varphi_\ast(I)(G/H)\cdot\varphi_\ast(J)(G/H)\subseteq
(\varphi_\ast(I)\cdot\varphi_\ast(J))(G/H)$, where the last inclusion is
\Cref{def:tambara-product}. By \Cref{prop:levelwise-radical},
$\varphi_H(x)\in\sqrt{\varphi_\ast(I)\cdot\varphi_\ast(J)}(G/H)$. Since the
latter is a Tambara ideal containing every generator of
$\varphi_\ast(I\cap J)$, we obtain
$\varphi_\ast(I\cap J)\subseteq\sqrt{\varphi_\ast(I)\cdot\varphi_\ast(J)}$,
and taking radicals gives $\subseteq$ in \eqref{eq:ext-meet}.
\end{proof}

\begin{Cor}\label{cor:functoriality-spectrum}
A morphism $\varphi\colon T\to S$ of $G$-Tambara functors induces a spectral
map $\varphi^a\colon\Spec_{\Nak}(S)\to \Spec_{\Nak}(T)$.
\end{Cor}

\begin{proof}
By \Cref{prop:extension-frame-map},
$\widetilde{\varphi}\colon\RadId_G(T)\to\RadId_G(S)$ is a frame homomorphism.
Applying the point functor and using \Cref{thm:points-primes} gives a continuous
map. To see that it is spectral, it suffices to show that
$\widetilde{\varphi}$ sends compact elements to compact elements. Let
$I\in\RadId_G(T)$ be compact. By \Cref{thm:coherent},
$I=\sqrt{\langle x_1,\dots,x_n\rangle}$ for some finite family
$x_i\in T(G/H_i)$. By
\Cref{lem:extension-radical-basic},
$\widetilde{\varphi}(I)=\sqrt{\varphi_\ast(\langle x_1,\dots,x_n\rangle)}$.
Since $\varphi$ commutes with all structure maps,
$\varphi_\ast(\langle x_1,\dots,x_n\rangle)
=\langle\varphi_{H_1}(x_1),\dots,\varphi_{H_n}(x_n)\rangle$,
so $\widetilde{\varphi}(I)
=\sqrt{\langle\varphi_{H_1}(x_1),\dots,\varphi_{H_n}(x_n)\rangle}
\in\RadFg_G(S)$, hence is compact by \Cref{thm:coherent}.
\end{proof}

\begin{Prop}\label{prop:points-are-preimages}
Under the identification of \Cref{thm:points-primes}, the map $\varphi^a$ of
\Cref{cor:functoriality-spectrum} sends $Q\in\Spec_{\Nak}(S)$ to its levelwise
preimage $\varphi^{-1}(Q)\in\Spec_{\Nak}(T)$.
\end{Prop}

\begin{proof}
By \Cref{rem:points-meet-prime}, the point of $\RadId_G(T)$ corresponding to
$\varphi^a(Q)$ is the meet-prime element
$\bigvee\{I\in\RadId_G(T)\mid\widetilde{\varphi}(I)\subseteq Q\}$. Since $Q$ is
radical,
\[
\widetilde{\varphi}(I)\subseteq Q
\iff
\sqrt{\varphi_\ast(I)}\subseteq Q
\iff
\varphi_\ast(I)\subseteq Q
\iff
I\subseteq\varphi^{-1}(Q).
\]
The middle equivalence uses that $Q$ is radical, so $\sqrt{\varphi_\ast(I)}\subseteq Q$
iff $\varphi_\ast(I)\subseteq Q$. For the last equivalence $\varphi_\ast(I)$ is
by definition the smallest Tambara ideal of $S$ containing the subsets
$\varphi_H(I(G/H))$ for all $H\le G$, so $\varphi_\ast(I)\subseteq Q$ if and only if
$\varphi_H(I(G/H))\subseteq Q(G/H)$ for every $H$, if and only if $I(G/H)\subseteq
\varphi_H^{-1}(Q(G/H))=\varphi^{-1}(Q)(G/H)$ for every $H$, if and only if
$I\subseteq\varphi^{-1}(Q)$. The join of all radical Tambara ideals contained
in $\varphi^{-1}(Q)$ is $\varphi^{-1}(Q)$ itself, since $\varphi^{-1}(Q)$ is
radical by \Cref{lem:preimage-radical}. Hence $\varphi^a(Q)=\varphi^{-1}(Q)$.
\end{proof}

\section{Consequences of the frame-theoretic perspective}\label{sec:consequences}

In this section we collect several structural results about the Nakaoka spectrum
that follow cleanly from the frame-theoretic description. We begin with the
following, which is \cite[Corollary~4.10]{Nakaoka2012Ideals}.

\begin{Prop}\label{prop:closed-immersion}
Let $I\trianglelefteq T$ be any Tambara ideal. The quotient map
$\pi\colon T\twoheadrightarrow T/I$ induces an isomorphism of frames
\[
\RadId_G(T/I)\cong {\uparrow}\sqrt{I}
\coloneqq \{J\in \RadId_G(T)\mid \sqrt{I}\subseteq J\}.
\]
Consequently, there is a homeomorphism
$\Spec_{\Nak}(T/I)\cong V(I)=V(\sqrt{I})
\coloneqq \{P\in\Spec_{\Nak}(T)\mid I\subseteq P\}$.
\end{Prop}

\begin{proof}
By \cite[Proposition~2.12]{Nakaoka2012Ideals}, the Tambara ideals of $T/I$ correspond
bijectively to the Tambara ideals of $T$ containing $I$, via $\bar J\mapsto
\pi^{-1}(\bar J)$. Under this correspondence, radical ideals match. Indeed, at
each level $G/H$, ring ideals $\bar J(G/H)\subseteq T(G/H)/I(G/H)$ correspond
to ring ideals $J(G/H)\supseteq I(G/H)$ in $T(G/H)$, and $\bar J(G/H)$ is
radical in $T(G/H)/I(G/H)$ if and only if $J(G/H)$ is radical in $T(G/H)$; combined
with \Cref{prop:levelwise-radical}, this yields the claim. A radical Tambara
ideal $J$ with $J\supseteq I$ automatically satisfies $J=\sqrt{J}\supseteq
\sqrt{I}$, so the correspondence restricts to a bijection between
$\RadId_G(T/I)$ and ${\uparrow}\sqrt{I}$. The correspondence preserves the
inclusion order, hence preserves meets (which are intersections on both sides)
and joins (which on both sides are the radical of the sum); therefore it is
an isomorphism of frames.

Taking points yields a homeomorphism
$\pt(\RadId_G(T/I))\cong\pt({\uparrow}\sqrt{I})$. By \Cref{thm:points-primes},
the left-hand side identifies with $\Spec_{\Nak}(T/I)$. The points of
${\uparrow}\sqrt{I}$ are precisely the prime Tambara ideals
$P\in\Spec_{\Nak}(T)$ such that $\sqrt{I}\subseteq P$ (equivalently, the
meet-prime elements of $\RadId_G(T)$ that lie in ${\uparrow}\sqrt{I}$), and this
is equivalent to $I\subseteq P$. Thus
$\pt({\uparrow}\sqrt{I})\cong V(I)$, and therefore
$\Spec_{\Nak}(T/I)\cong V(I)$.
\end{proof}

\begin{Def}\label{def:nilradical}
The \emph{nilradical} of a $G$-Tambara functor $T$ is
$\mathfrak{N}(T)\coloneqq \sqrt{0}$, the radical of the zero ideal. By
\Cref{prop:levelwise-radical},
$\mathfrak{N}(T)(G/H)=\mathrm{nil}(T(G/H))$ is the ordinary nilradical of
$T(G/H)$ for each $H\le G$, and
$\mathfrak{N}(T)=\bigcap_{P\in\Spec_{\Nak}(T)}P$.
\end{Def}

\begin{Prop}\label{prop:nilradical-prime-iff-irreducible}
Suppose $T\ne 0$. The nilradical $\mathfrak{N}(T)=\sqrt{0}$ is a prime Tambara
ideal if and only if $\Spec_{\Nak}(T)$ is irreducible.
\end{Prop}

\begin{proof}
Since $T\ne 0$, $\mathfrak{N}(T)\ne T$, so the meet-prime condition on
$\mathfrak{N}(T)$ is meaningful. By \Cref{thm:points-primes},
$\mathfrak{N}(T)$ is Nakaoka-prime if and only if it is meet-prime in
$\RadId_G(T)$. Since $\RadId_G(T)$ is spatial
(\Cref{thm:spatiality}) with bottom element $\mathfrak{N}(T)$, it remains to
interpret meet-primality of the bottom element topologically. Under the
identification $\pt(\RadId_G(T))\cong \Spec_{\Nak}(T),$
the bottom element corresponds to the empty open subset. Thus $\bot$ is meet-prime precisely when $U\cap V=\varnothing \Longrightarrow U=\varnothing \text{ or } V=\varnothing$ for opens $U,V\subseteq \Spec_{\Nak}(T)$, which is exactly irreducibility.
\end{proof}

\begin{Def}\label{def:reduced}
A $G$-Tambara functor $T$ is \emph{reduced} if $\mathfrak{N}(T)=0$, i.e.\ if
each ring $T(G/H)$ has no nonzero nilpotent elements. The \emph{reduction} of
$T$ is $T_{\mathrm{red}}\coloneqq T/\mathfrak{N}(T)$.
\end{Def}

\begin{Prop}\label{prop:reduction-homeomorphism}
The projection $\pi\colon T\twoheadrightarrow T_{\mathrm{red}}$ induces an
isomorphism of frames $\RadId_G(T_{\mathrm{red}})\cong\RadId_G(T)$ and a
homeomorphism $\Spec_{\Nak}(T_{\mathrm{red}})\cong\Spec_{\Nak}(T)$.
\end{Prop}

\begin{proof}
Every radical Tambara ideal $I$ contains $\mathfrak{N}(T)$. Indeed, by
\Cref{prop:levelwise-radical},
\[
I=\sqrt{I}=\bigcap_{P\supseteq I}P\supseteq \bigcap_{\text{all }P}P
=\mathfrak{N}(T).
\]
Thus $\RadId_G(T)={\uparrow}\mathfrak{N}(T)$, and the result follows from
\Cref{prop:closed-immersion} applied to $I=\mathfrak{N}(T)$.
\end{proof}

\begin{Rem}
This recovers \cite[Corollary~4.13]{Nakaoka2012Ideals}. From the frame-theoretic
perspective, passage to the reduction is invisible; every radical Tambara ideal
contains $\mathfrak{N}(T)$, so the frame $\RadId_G(T)$ depends only on
$T_{\mathrm{red}}$.
\end{Rem}

\begin{Prop}[The Chinese remainder theorem]\label{prop:coprime-decomposition}
Let $I,J\in\RadId_G(T)$ be radical Tambara ideals with $I\vee J=T$ and
$I\wedge J=\mathfrak{N}(T)$. Equivalently, for every $H\le G$,
$I(G/H)+J(G/H)=T(G/H)$ and
$I(G/H)\cap J(G/H)=\mathrm{nil}(T(G/H))$.
Then:
\begin{enumerate}
\item[\textup{(i)}] $I$ is a complemented element of $\RadId_G(T)$, with
complement $J$.
\item[\textup{(ii)}] There is a homeomorphism
$\Spec_{\Nak}(T)\cong V(I)\sqcup V(J)$,
where $V(I)$ and $V(J)$ are clopen.
\item[\textup{(iii)}] There is an isomorphism of Tambara functors
$T_{\mathrm{red}}\cong (T/I)\times (T/J)$.
\end{enumerate}
\end{Prop}

\begin{proof}
First note that the levelwise formulation in the statement is equivalent to the
frame-theoretic one. Since $I$ and $J$ are radical,
\Cref{prop:radical-product-intersection-general} gives
$I\wedge J=I\cap J$, so
\[
I\wedge J=\mathfrak{N}(T)
\iff
I(G/H)\cap J(G/H)=\mathrm{nil}(T(G/H))
\quad\text{for all }H\le G.
\]
Also, $I\vee J=T$ means $\sqrt{I+J}=T$. By \Cref{prop:levelwise-radical}, this
is equivalent to
\[
\sqrt{I(G/H)+J(G/H)}=T(G/H)
\quad\text{for all }H\le G,
\]
which in turn is equivalent to
\[
I(G/H)+J(G/H)=T(G/H)
\quad\text{for all }H\le G.
\]

For (i), note that $\mathfrak{N}(T)$ is the bottom element of $\RadId_G(T)$, so
the assumptions say exactly that $I\vee J=\top$ and $I\wedge J=\bot$. Thus $I$
and $J$ are complements.

For (ii), a complemented element in a spatial frame determines a clopen
decomposition of the point space. Explicitly, $V(I)$ and $V(J)$ are closed by
definition. They are disjoint since $I\vee J=T$, so no prime can contain both.
They cover $\Spec_{\Nak}(T)$ because every prime contains
$\mathfrak{N}(T)=I\wedge J$, and meet-primality then forces it to contain
$I$ or $J$.

For (iii), by the equivalence just established, $I+J=T$. Thus $I$ and $J$ are
coprime in the sense of \cite[Section~3]{Nakaoka2012Ideals}. By
\cite[Corollary~3.19]{Nakaoka2012Ideals}, $IJ=I\cap J=\mathfrak{N}(T)$, and
there is an isomorphism of $G$-Tambara functors
\[
T_{\mathrm{red}}=T/\mathfrak{N}(T)=T/(IJ)\xrightarrow{\cong}
(T/I)\times(T/J).\qedhere
\]
\end{proof}

\begin{Rem}
\Cref{prop:coprime-decomposition} is essentially a repackaging of
\cite[Corollary~3.19]{Nakaoka2012Ideals} from the frame-theoretic perspective.
Nakaoka's result applies to any coprime pair of ideals and produces
$T/(I\cap J)\cong (T/I)\times(T/J)$; the additional hypothesis
$I\wedge J=\mathfrak{N}(T)$ is what replaces $T/(I\cap J)$ by
$T_{\mathrm{red}}$ and turns the statement into one about complemented
elements of the frame $\RadId_G(T)$.
\end{Rem}

\begin{Cor}\label{cor:connectedness}
The following are equivalent:
\begin{enumerate}
\item[\textup{(i)}] $\Spec_{\Nak}(T)$ is disconnected.
\item[\textup{(ii)}] $\RadId_G(T)$ has a nontrivial complemented element.
\item[\textup{(iii)}] There exist coprime radical Tambara ideals
$I,J\trianglelefteq T_{\mathrm{red}}$ with $I\cap J=0$ and $I,J\ne 0$.
\item[\textup{(iv)}] There exist nontrivial $G$-Tambara functors $T_1,T_2$
with $T_{\mathrm{red}}\cong T_1\times T_2$.
\end{enumerate}
\end{Cor}

\begin{proof}
(i)$\Leftrightarrow$(ii): By \Cref{cor:frame-opens},
$\RadId_G(T)\cong \Omega(\Spec_{\Nak}(T))$. A complemented element of
$\Omega(\Spec_{\Nak}(T))$ is exactly a clopen subset. To see this, note if $U$ has complement $V$, then $U\cap V=\varnothing$ and $U\cup V=\Spec_{\Nak}(T)$, so $U$ is clopen.
Conversely, a clopen subset has open complement and is therefore complemented.
Thus $\Spec_{\Nak}(T)$ is disconnected if and only if $\RadId_G(T)$ has a nontrivial
complemented element.

(ii)$\Leftrightarrow$(iii): By \Cref{prop:reduction-homeomorphism}, we may
replace $T$ by $T_{\mathrm{red}}$. Since $T_{\mathrm{red}}$ is reduced, its
nilradical is zero, so the bottom element of $\RadId_G(T_{\mathrm{red}})$ is
$0$. A pair of nontrivial complements is exactly a pair of radical Tambara
ideals $I,J$ with $I\vee J=T_{\mathrm{red}}$, $I\wedge J=0$, and $I,J\neq 0$.
By \Cref{prop:radical-product-intersection-general}, $I\wedge J=I\cap J$ for
radical ideals, so $I\wedge J=0$ is equivalent to $I\cap J=0$, and
$I\vee J=T_{\mathrm{red}}$ is equivalent to $I+J=T_{\mathrm{red}}$ by the same
argument as in \Cref{prop:coprime-decomposition}, so $I$ and $J$ are coprime.

(iii)$\Leftrightarrow$(iv): The forward direction follows from
\Cref{prop:coprime-decomposition}(iii). For the converse, suppose
$T_{\mathrm{red}}\cong T_1\times T_2$ with $T_1,T_2\ne 0$. Let
$I=\ker(T_{\mathrm{red}}\twoheadrightarrow T_1)$ and
$J=\ker(T_{\mathrm{red}}\twoheadrightarrow T_2)$; at each level $G/H$ these
are respectively $0\times T_2(G/H)$ and $T_1(G/H)\times 0$ in
$T_1(G/H)\times T_2(G/H)$. They are Tambara ideals with $I\cap J=0$ and
$I+J=T_{\mathrm{red}}$. Since $T_{\mathrm{red}}$ is reduced, each $T_i(G/H)$
is reduced, and therefore each $I(G/H)$ and $J(G/H)$ is a radical ring ideal;
by \Cref{prop:levelwise-radical}, $I$ and $J$ are radical Tambara ideals.
They are nonzero since $T_1,T_2\ne 0$.
\end{proof}

\begin{Rem}
\Cref{cor:connectedness} recovers part of
\cite[Proposition~4.15]{Nakaoka2012Ideals}. The equivalence
(i)$\Leftrightarrow$(ii) is purely frame-theoretic and avoids any direct
prime-ideal argument.
\end{Rem}

\bibliographystyle{alpha}
\bibliography{bib}

\end{document}